\documentclass[12pt]{elsarticle}
\usepackage[left=1in,right=1in, top=1.2in,bottom=1.2in]{geometry}
\usepackage{amssymb,amsthm,latexsym,amsmath,epsfig,pgf}
\usepackage{tikz}

\newtheorem{observation}{Observation}[section]
\newtheorem{theorem}{Theorem}[section]
\newtheorem*{theorem A}{Theorem A}
\newtheorem*{theorem B}{N\"olker's Theorem}
\newtheorem{lemma}{Lemma}[section]

\newtheorem{corollary}{Corollary}[section]
\newtheorem{definition}{Definition}[section]
\newtheorem{problem}{Problem}

\newtheorem {conjecture}{Conjecture}[section]
\theoremstyle{remark}
\newtheorem{remark}{Remark}[section]
\theoremstyle{remark}

\begin{document}

\begin{frontmatter}
\title{$D$-Antimagic Labelings\\ on Oriented Linear Forests}


\author[label1]{Ahmad Muchlas Abrar}
\author[label2,label3]{Rinovia Simanjuntak $^{*,}$}

\address[label1]{\small Doctoral Program in Mathematics, Faculty of Mathematics and Natural Sciences, \\Institut Teknologi Bandung, Jalan Ganesa 10 Bandung, Indonesia}
\address[label2]{\small Combinatorial Mathematics Research Group,
Faculty of Mathematics and Natural Sciences, \\ Institut Teknologi Bandung, Jalan Ganesa 10 Bandung, Indonesia}
\address[label3]{\small Centre for Research Collaboration in Graph Theory and Combinatorics, Indonesia

\vspace*{2.5ex} 
 {\normalfont 30123302@mahasiswa.itb.ac.id, rino@itb.ac.id\\}
\vspace*{2.5ex} 
 {\normalfont $^*$\textit{corresponding author}}
 }

\begin{abstract}
Let $\overrightarrow{G}$ be an oriented graph with the vertex set $V(\overrightarrow{G})$ and the arc set $A(\overrightarrow{G})$. Suppose that $D\subseteq \{0,1,2,\dots,\partial \}$ is a distance set where $\partial=\max \{d(u,v)<\infty|u,v\in V(\overrightarrow{G})\}$. Given a bijection  $h:V(\overrightarrow{G}) \rightarrow\{1,2,,\dots,|V(\overrightarrow{G})|\}$, the $D$-weight   of a vertex $v\in V(\overrightarrow{G})$ is defined as $\omega_D(v)=\sum_{u\in N_D(v)}h(u)$, where $N_D(v)=\{u\in V|d(v,u)\in D\}$. A bijection $h$ is called a $D$-antimagic labeling if for every pair of distinct vertices $x$ and $y$, $\omega_D(x)\ne \omega_D(y)$. An oriented graph $\overrightarrow{G}$ is called $D$-antimagic if it admits such a labeling. 

In addition to introducing the notion of $D$-antimagic labeling for oriented graphs, we investigate some properties of $D$-antimagic oriented graphs. In particular, we study $D$-antimagic linear forests for some $D$. We characterize $D$-antimagic paths where $1 \in D$, $n-1\in D$, or $\{0,n-2\}\subset D$. We characterize distance antimagic trees and forests. We conclude by constructing $D$-antimagic labelings on oriented linear forests.

\let\thefootnote\relax\footnotetext{Received: xx xxxxx 20xx,\quad
  Accepted: xx xxxxx 20xx.\\[3ex]
  }
  
\end{abstract}

\begin{keyword}
$D$-antimagic labeling \sep oriented graph \sep oriented path \sep oriented tree \sep oriented linear forest

Mathematics Subject Classification : 05C78, 05C12

\end{keyword}

\end{frontmatter}

\section{Introduction}

Let $\overrightarrow{G}$ be a finite, simple, and oriented graph with the vertex set $V(\overrightarrow{G})$ and the arc set $A(\overrightarrow{G})$. For $u,v\in V(\overrightarrow{G})$, \textit{the arc from $u$ to $v$} is denoted as $(u,v)$. \textit{The in-neighbors of $u$} are all vertices $v$ such that $(v,u)\in A(\overrightarrow{G})$ and \textit{the out-neighbors of $u$} are all vertices $v$ such that $(u,v)\in A(\overrightarrow{G})$. \textit{The in-degree of $u$}, denoted $d^-(u)$, is the number of in-neighbors of $u$ and \textit{the out-degree of $u$}, denoted $d^+(u)$, is the number of out-neighbors of $u$. A vertex with zero in-degree is called a \textit{source}, while a vertex with zero out-degree is called a \textit{sink}. 
\textit{The distance from $u$ to $v$}, denoted $d(u,v)$, is the length of a shortest directed path from $u$ to $v$, if it exists; otherwise, $d(u,v)=\infty$. $\overrightarrow{G}$ is \textit{connected} if for each pair of distinct vertices $u$ and $v$, either $d(u,v)$ or $d(v,u)$ is finite. If both distances are finite, $\overrightarrow{G}$ is classified as \textit{strongly connected}.

Since each arc in an oriented graph contributes exactly one in-degree and one out-degree, we have the following: 
\begin{lemma} [Handshake Lemma for Oriented Graphs]  \cite{DB West} \label{lem:handshake}
If $\overrightarrow{G}$ is an oriented graph and $u\in V(\overrightarrow{G})$, 
then
$$\sum_{u\in V(\overrightarrow{G})}^{u}{d^{-}(u)}=\sum_{u\in V(\overrightarrow{G})}^{u}{d^{+}(u)}=|A(\overrightarrow{G})|.$$
\end{lemma}
    
In 2013, Kamatchi and Arumugam \cite{Kamatchi-64-13} introduced the concept of distance antimagic labeling for undirected graphs. Let $G=(V,E)$ be an undirected and simple graph of order $n$. Let $h:V(G)\rightarrow \{1,2,\dots,n\}$ be a bijection.  \textit{The weight of a vertex $u$} is defined as $\omega(u)=\sum_{v\in N(u)}{h(v)}$, where $N(u)$ is the neighbourhood of $u$. If $\omega(x)\ne \omega(y)$ for any distinct vertices $u$ and $v$, then $h$ is called a \textit{distance antimagic labeling} and $G$ is said to be \textit{distance antimagic}. Several classes of graphs have been proven to be distance antimagic, which include paths $P_n$, cycles $C_n \ (n\ne 4)$, and wheels $W_n \ (n\ne 4)$ \cite{Kamatchi-64-13}, as well as hypercube $Q_n \ (n\ge 4)$ \cite{Kamatchi cube}, and trees with $l$ leaves and $2l$ vertices \cite{Llado}. Furthermore, some graphs resulting from the cartesian, strong, direct, lexicographic, corona, and join products have also been shown to be antimagic distance \cite{simanjuntakprod,Handa}. These results led to the following conjecture of Kamatchi and Arumugam \cite{Kamatchi-64-13}.

\begin{conjecture} \cite{Kamatchi-64-13} \label{conj:DA}
A graph $G$ is distance antimagic if and only if there are no two distinct vertices with identical neighborhoods.
\end{conjecture}

Conjecture \ref{conj:DA} has been computationally verified for all graphs of order up to 8 \cite{AAC}. In the same paper, Simanjuntak \textit{et. al.} generalized the concept of distance antimagic labeling to $D$-antimagic labeling. Let $G$ be a graph and $D\subseteq \{0,1,2,\dots, diam(G)\}$ be a non-empty distance set. For a vertex $u$ in $G$, $N_D(u)=\{v| d(v,u)\in D\}$ is the \textit{$D$-neighborhood of $u$}. A \textit{$D$-antimagic labeling of $G$} is a bijection $f:V(G)\rightarrow \{0,1,2,\dots, diam(G)\}$ such that the $D$-weight $\omega_D(u)=\sum_{v\in N_D(u)}f(v)$ is distinct for every vertex $u$. In this case, $G$ is called \textit{$D$-antimagic.} Note that the distance antimagic labeling is a $\{1\}$-antimagic labeling. 
Thus, a generalization to Conjecture \ref{conj:DA} was also proposed.
\begin{conjecture} \cite{AAC}
A graph $G$ is $D$-antimagic if and only if each vertex in $G$ has distinct $D$-neighborhood.
\end{conjecture}

Although much research has focused on antimagic properties in undirected graphs, the corresponding study for directed graphs, where the direction of arcs adds a layer of complexity, is still in its infancy. This paper aims to explore this less studied area and bridge that gap by extending the concept of $D$-antimagic labeling to oriented graphs, specifically in the context of linear forest graphs. First, we naturally generalize the definition of $D$-antimagic labeling for undirected graphs and define it for oriented graphs.

\begin{definition}
Let $\overrightarrow{G}$ be an oriented graph, $\partial=\max\{d(u,v)<\infty|u,v\in V(\overrightarrow{G})\}$, and $D\subseteq \{0,1,2,\dots,\partial\}$ be a non-empty distance set.
A \textbf{$D$-antimagic labeling of $\overrightarrow{G}$} is a bijection $f:V(\overrightarrow{G}) \rightarrow\{1,2,\dots,|V(\overrightarrow{G})|\}$ such that for any pair of distinct vertices $u$ and $v$, $\omega_D(u) \neq \omega_D(v)$, with $\omega_D(v)=\sum_{u\in N_D(v)}f(u)$ is the \textbf{$D$-weight of a vertex $v$}, where $N_D(v)=\{u|d(v,u)\in D\}$.  In this case, $\overrightarrow{G}$ is called \textbf{$D$-antimagic}.
\end{definition}
Note that when $\overrightarrow{G}$ is strongly connected, then $\partial$ is the diameter of $\overrightarrow{G}$.

This paper studies $D$-antimagic linear forests for some $D$. In particular, we characterize $D$-antimagic paths where $1 \in D$, $n-1\in D$, or $\{0,n-2\}\subset D$ (Section \ref{sec:path}). We characterize distance antimagic trees and forests in Sections \ref{sec:tree} and \ref{sec:forest}. We conclude by constructing $D$-antimagic labelings on oriented linear forests in Section \ref{sec:forest}. 


\section{General Results} \label{sec:general}

We start by considering the "trivial" $D$-antimagic labeling.
\begin{observation}\label{obs:0-antimagic}
All oriented graphs are $\{0\}$-antimagic.
\end{observation}



Let $\overrightarrow{G}$ be strongly connected with diameter $\partial$. For any vertex $u$ in $\overrightarrow{G}$, the vertex set $V(\overrightarrow{G})$ can be partitioned based on the distance of other vertices to $u$, that is
   \[V(\overrightarrow{G})=\bigcup_{t=0}^{\partial}{N_{\{t\}}(u)},\]
where $N_{\{t_1\}}(u)\cap N_{\{t_2\}}(u)=\emptyset, t_1\neq t_2$. This leads to the following theorem.
\begin{theorem} \label{thm:D*antimagic}
Let $\overrightarrow{G}$ be strongly connected and $D\subset \{0,1,\dots,diam(\overrightarrow{G})\}$ be a distance set. If $D^*=\{0,1,\dots,diam(\overrightarrow{G})\}\backslash D$, then $\overrightarrow{G}$ is $D$-antimagic if and only if $\overrightarrow{G}$ is $D^*$-antimagic.
\end{theorem}
\begin{proof}
Let $f$ be a $D$-antimagic labeling of $\overrightarrow{G}$.   Since adjacency is mutually exclusive, the $D^*$-weight of $u$, $\omega_{D^*}(u)=\binom{|V(G)|}{2}-\omega_D(u)$. Since all $D$-weights are different, all $D^*$-weights are also different. We complete the proof by swapping the roles of $D$ and $D^*$.
\end{proof}


The following corollary directly follows from Observation \ref{obs:0-antimagic} and Theorem \ref{thm:D*antimagic}.
\begin{corollary}
    Let $d$ be a diameter of $\overrightarrow{G}$. Any strongly connected $\overrightarrow{G}$ is $\{1,2,\dots,d\}$-antimagic and not  $\{0,1,2,\dots,d\}$-antimagic.
\end{corollary}

Although this paper focuses on $D$-antimagic labelings, we could also prove a corresponding result to Theorem \ref{thm:D*antimagic} for a magic labeling version that Marr and Simanjuntak \cite{marr} introduced. A graph $G$ is \textit{$D$-magic} if it admits a bijection $f:V(\overrightarrow{G}) \rightarrow\{1,2,\dots,|V(\overrightarrow{G})|\}$ such that the $D$-weights of each vertex is equal to $\lambda$. In such a case, $\lambda$ is called a \textit{magic constant}. 
\begin{theorem} \label{thm:iffmagic}
Let $\overrightarrow{G}$ be strongly connected, $D\subset \{0,1,\dots,diam(\overrightarrow{G})\}$ be a non-empty distance set, and $D^*=\{0,1,\dots,diam(\overrightarrow{G})\}\backslash D$. Then $\overrightarrow{G}$ is $D$-magic with magic constant $\lambda$ if and only if $\overrightarrow{G}$ is $D^*$-magic with magic constant $\lambda^*=\binom{V(\overrightarrow{G})}{2}-\lambda.$
 \end{theorem}
\begin{proof}
Suppose that $f$ is a $D$-magic labeling on $\overrightarrow{G}$ such that $\omega_D(u)=\lambda$ holds for every $u\in V(\overrightarrow{G})$. 
Since $\overrightarrow{G}$ is strongly connected, for any vertex $u\in \overrightarrow{G}$, $N_{D^*}(u)=V(\overrightarrow{G})\backslash N_D(u)$. And so $\omega_{D^*}(u)=\binom{V(\overrightarrow{G})}{2}-\lambda$  for every $u\in V(\overrightarrow{G})$. Hence, $f$ is also a magic $D^*$-labeling on $\overrightarrow{G}$. Switching between $D$ and $D^*$, we complete the proof.
\end{proof}

In the same paper, Marr and Simanjuntak \cite{marr} also provided lower and upper bounds for the magic constant of $D$-magic graphs, where the lower bound is sharp, while the upper bound is not. They then asked for a sharp upper bound, which for strongly connected graphs can be answered as a direct consequence of Theorem \ref{thm:iffmagic}.

\begin{theorem}\cite{marr}\label{thm:ub}
Let $\overrightarrow{G}$ be an oriented graph of order $n \ge 3$. If $\overrightarrow{G}$ is $D$-magic with magic constant $\lambda$ then $5 \le \lambda \le \frac{n(n+1)}{2}$.
\end{theorem}

\begin{corollary}     
Let $\overrightarrow{G}$ be a strongly connected graph of order $n \ge 3$. If $\overrightarrow{G}$ is $D$-magic with magic constant $\lambda$, then $5 \le \lambda \le \frac{n(n+1)}{2}-5$.
\end{corollary}
\begin{proof}
Assume that $\lambda>\frac{n(n+1)}{2}-5$ by a $D$-magic labeling $f$. By Theorem \ref{thm:iffmagic}, there exists a $D^*$-magic labeling $f^*$ with $\lambda^*<5$. This contradicts Theorem \ref{thm:ub}. Thus, $\lambda\le \frac{n(n+1)}{2}-5$.
Figure \ref{fig:DM} confirms the upper bound is sharp.
\end{proof}
\begin{figure}[h!]
    \centering    \includegraphics[width=0.3\linewidth]{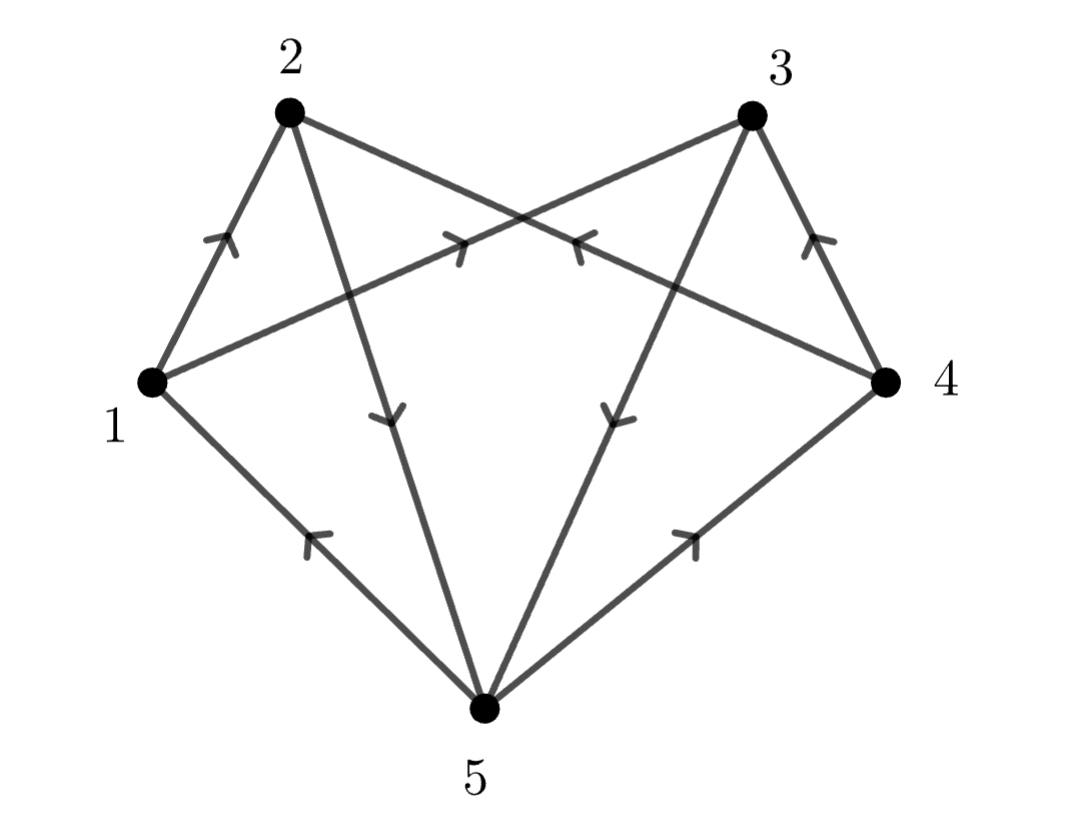}
    \caption{A $\{0,2,3\}$-magic graph of order $5$ with magic constant $10$.}
    \label{fig:DM}
\end{figure}

Another natural question is whether a graph that is both $D_1$ and $D_2$-antimagic is also $(D_1\cup D_2$)-antimagic. Consider the cycle $\overrightarrow{C_4}$ with unidirectional orientation with
$D_1=\{0\}$ and $D_2=\{2\}$. By Observation \ref{obs:0-antimagic}, $\overrightarrow{C_4}$ is $\{0\}$-antimagic, and by a bijection that assigns the vertices to $1,2,3,4$ cyclically, $\overrightarrow{C_4}$ is also $\{2\}$-antimagic. However, for $D=D_1\cup D_2=\{0,2\}$, two vertices of distance 2 have the same $D$-neighborhood, and so $\overrightarrow{C_4}$ is not $D$-antimagic. The undirected cycle $C_4$ also serves as a counterexample for the undirected version of the question. 

\begin{remark} 
 A (oriented or not) $D_1$ and $D_2$-antimagic graph, is not necessarily $(D_1\cup D_2)$-antimagic. 
\end{remark}


Finally, we consider a forbidden orientation for a distance antimagic graph, where the existence of sinks is essential. 

\begin{theorem} 
Let $\overrightarrow{G}$ be an oriented graph. If $\overrightarrow{G}$ has more than one sink or exactly one sink with two or more in-neighbors of out-degree one, then $\overrightarrow{G}$ is not distance antimagic. 
\end{theorem} 
\begin{proof} 
Assume that $\overrightarrow{G}$ has more than one sink, say $s_1$ and $s_2$. Since $N(s_1)=N(s_2)=\emptyset$, then for any bijection, $\omega(s_1)=\omega(s_2)=0$. 
Now assume that $\overrightarrow{G}$ has exactly one sink, say $s_1$, with two or more in-neighbors of out-degree one, say $u_1$ and $u_2$. Consequently, for any bijection $f$, $\omega(s_1)=\omega(s_2)=f(s_1)$. 
\end{proof}


\section{$D$-Antimagic Labeling on Oriented Paths} \label{sec:path} 

We start our observation with the simplest tree, that is the path. It is worth mentioning that, for any orientation, a path of two or more vertices has at least one sink and one source.
\begin{lemma} \label{lem:sinksource} For $n\ge 2$, any oriented path $\overrightarrow{P_n}$ has at least one sink and one source. \end{lemma}
\begin{proof} By Lemma \ref{lem:handshake}, the sum of in-degrees and the sum of out-degrees in any oriented graph must be equal. Note that each internal vertex that is neither a sink nor a source will increase the number of in-degrees and out-degrees by one each simultaneously, so they can be overlooked. Thus, we only need to consider the possible sinks and sources. There are only three possible cases to keep the number of in-degrees and out-degrees equal.
\begin{description}
    \item[Case 1] If both end vertices are sinks, then in addition to the pairs of internal sink and source (if any), there must be another internal vertex that is a source.
    \item[Case 2] If both end vertices are sources, then in addition to the pairs of internal sink and source (if any), there must be another internal vertex that is a sink.
    \item[Case 3] If the end vertices are already a pair of one sink and one source, then the number of sinks and sources are equal. 
\end{description}
\end{proof}

Thus, we obtain the following characteristics of an orientation on a path. 
\begin{remark}\label{rem:oripath} For $n\geq 2$, an oriented path $\overrightarrow{P_n}$ satisfies one of the following. 
\begin{enumerate} 
\item If both end vertices are sinks, then the number of sources is one less than the number of sinks.
\item If both end vertices are sources, then the number of sinks is one less than the number of sources.
\item If one end vertex is a sink and the other a source,  then the number of sinks equals the number of sources. In particular, if the path has only one sink and one source then it is unidirectional.
 \end{enumerate} 
 \end{remark} 
 

It is obvious that the largest distance in an oriented path of order $n$ is $n-1$, which only happens when its orientation is unidirectional and its diameter is $n-1$. Thus, the maximum value in the distance set $D$ for $D$-antimagic oriented path of order $n$ is $n-1$. In the next lemma, we study the minimum value in $D$.
\begin{lemma} \label{lem:minD}
Let $n\geq 3$. If an oriented path $\overrightarrow{P_n}$ is $D$-antimagic then $\min{(D)}\le 1$.
\end{lemma}
\begin{proof}
Assume that $min(D) \ge 2$ and let $D \subset\{0, 1, 2, \ldots, n-1\}$ with $\min(D) \ge 2$. By Lemma \ref{lem:sinksource}, $\overrightarrow{P_n}$ has at least one sink, thus we observe two cases based on the number of sinks. 

\begin{description}
    \item[Case 1 ($\overrightarrow{P_n}$ has more than one sink).] Here, at least two sinks have empty $D$-neighbourhoods, which results in zero $D$-weights of both sinks under any vertex labeling. 
    \item[Case 2 ($\overrightarrow{P_n}$ has exactly one sink).] We shall consider two subcases based on the location of the sink.
    \begin{description}
        \item[Case 2.1 (The sink is a leaf).] Here, the orientation is unidirectional. Since $\min(D)\ge 2$, the sink and its in-neighbor have empty $D$-neighborhoods, thus the same $D$-weights. 
        \item[Case 2.2 (The sink is an internal vertex).] Consider the sink and its two in-neighbors. Since $min(D) \ge 2$, these three vertices have empty $D$-neighborhoods, thus the same $D$-weights. 
    \end{description}
\end{description}
All the cases considered lead to contradictions, and so $\min{(D)}\le 1$.
\end{proof}

In the next theorem, we show that the necessary condition for the existence of $D$-antimagic paths in Lemma \ref{lem:minD} is sufficient if the orientation is unidirectional.

\begin{theorem}\label{th:uni->minD<=1}
Let $n\geq 3$ and $D$ be a distance set with $\min{(D)}\le 1$. If $\overrightarrow{P_n}$ is unidirectional, then $\overrightarrow{P_n}$ is $D$-antimagic.  
\end{theorem}
\begin{proof} Let $V(\overrightarrow{P_n})=\{v_1,v_2, \ldots, v_n\}$, $A(\overrightarrow{P_n})=\{(v_i,v_{i+1})|1\le i\le n-1\}$, and distance set $D=\{d_0, d_1,d_2,\ldots, d_k\}$, where $0\le d_0<d_1<d_2<\ldots<d_k\le n-1$. Define a bijection $g:V(\overrightarrow{P_n})\rightarrow\{1, 2, \ldots, n\}$, with $g(v_i)=n-i+1$, for $i=1, 2, \ldots, n$. We shall prove that $g$ is $D$-antimagic, by considering two cases based on  $\min{(D)}=d_0$.
\begin{description}
    \item[Case 1 ($d_0=0$)] Here, the $D$-neighborhoods and $D$-weights of the vertices in $\overrightarrow{P_n}$ are as follows.\\ 
    \textbf{For $1 \le i \le n-d_k$,} $N_D(v_i) = \{v_{i+d_t}|0 \le t \le k\}$ and $\omega_D(v_i)
    =(k+1)(n-i+1)-\sum_{t=0}^kd_t$,\\
    \textbf{For $n-d_{t_0+1}+1 \le i \le n-d_{t_0}, t_0 \in [1,k-1]$,}   $N_D(v_i) = \{v_{i+d_t}|0 \le t \le t_0\}$ and $\omega_D(v_i)
    =(t_0+1)(n-i+1)-\sum_{t=0}^{t_0} d_t$,\\
    \textbf{For $n-d_1+1\le i \le n$,} $N_D(v_i) = \{v_i\}$ and $\omega_D(v_i)=n-i+1$.\\
    Thus, the $D$-weights assemble a strictly decreasing sequence, that is, $\omega_D(v_i)>\omega_D(v_{i+1})$, for $1\le i\le n-1$.
    \item[Case 2 ($d_0=1$)]  For $D=\{1\}$, $N_{\{1\}}(v_n)=\emptyset$ and $N_{\{1\}}(v_i)=\{v_i+1\}$, for $1\le i\le n-1$. And so, all the $\{1\}$-weights are different, where $\omega_{\{1\}}(v_n)<\omega_{\{1\}}(v_i)<\omega_{\{1\}}(v_j)$, for $i<j<n$.\\ 
    For $|D|\ge 2$, the $D$-neighborhoods and $D$-weights of the vertices in $\overrightarrow{P_n}$ are as follows.\\
    \textbf{For $1 \le i \le n-d_k$,} $N_D(v_i) = \{v_{i+d_t}|0 \le t \le d_k\}$ and $\omega_D(v_i)
 =(k+1)(n-i+1)-\sum_{t=0}^kd_t$. \\ 
    \textbf{For $n-d_{t_0+1}+1\le i \le n-d_{t_0}, t_0 \in [1,k-1]$,} $N_D(v_i) = \{v_{i+d_t}|0 \le t \le t_0\}$ and $\omega_D(v_i)
 =(t_0+1)(n-i+1)-\sum_{t=0}^{t_0} d_t$.\\  
    \textbf{For $n-d_1+1\le i \le n$,} $N_D(v_i) = \{v_{i+1}\}$ and $N_D(v_n) = \emptyset$ and $\omega_D(v_i)
 =n-i$ and $\omega_D(v_n)=0$.\\
 Thus, $\omega_D(v_i)<\omega_D(v_{i+1})$, for $1\le i\le n-1$. 
 \end{description}

\end{proof}

In the next theorem, we characterize $D$-antimagic oriented paths where $\min(D)=1$.

\begin{theorem}\label{th:minD=1}
Let $n\ge3$ and $D$ be a distance set with $\min(D)=1$. Then $\overrightarrow{P_n}$ is $D$-antimagic if and only if $\overrightarrow{P_n}$ is unidirectional.
\end{theorem}
\begin{proof} Let $\overrightarrow{P_n}$ be a $D$-antimagic graph. Assume that $\overrightarrow{P_n}$ is not unidirectional and consider the three cases in Remark \ref{rem:oripath}. 
\begin{description}
    \item[Case 1 (Both end vertices are sinks)] The two sinks have zero $D$-weights.
    \item[Case 2 (Both end vertices are sources)] Consider two subcases based on the number of sinks. 
    \begin{description}
        \item[Case 2.1 ($\overrightarrow{P_n}$ has exactly one sink)] The two in-neighbors of the sink have the same $D$-neighborhood, that is the set contains only the sink.
        \item[Case 2.1 ($\overrightarrow{P_n}$ has two or more sinks)] Both sinks have zero $D$-weights.
    \end{description}
    \item[Case 3 (The two end vertices are a pair of one sink and one source)] Since $P_n$ is not unidirectional, there are at least two sinks of zero $D$-weights.
\end{description}
All these cases lead to contradictions. Thus, $\overrightarrow{P_n}$ is unidirectional. Conversely, Theorem \ref{th:uni->minD<=1} completes the proof.\end{proof}

To conclude this section, we investigate $D$-antimagic labelings of paths where the set of distances $D$ contains the two largest distances, that is, $n-1$ and $n-2$. 
Note that, $n-1$ is in $D$ only if the orientation is unidirectional. This and Theorem \ref{th:uni->minD<=1} lead to the following characterization.
\begin{theorem} Let $n\ge 3$ and $n-1\in D$.
    $\overrightarrow{P_n}$ is $D$-antimagic if only if $\overrightarrow{P_n}$ is unidirectional.
\end{theorem}  

For $D$ containing $n-2$, we define the following two orientations $\Theta'$ and $\Theta''$ (see Figure \ref{fig:theta}).
\begin{align*}
    \Theta'&:=\{(v_2,v_1)\}\cup\{(v_i,v_{i+1})|2\le i\le n-1\},\\
    \Theta''&:=\{(v_1,v_2)\}\cup\{(v_{i+1},v_{i})|2\le i\le n-1\}.
\end{align*} 
\begin{figure}[h!]
    \centering
    \includegraphics[scale=.45]{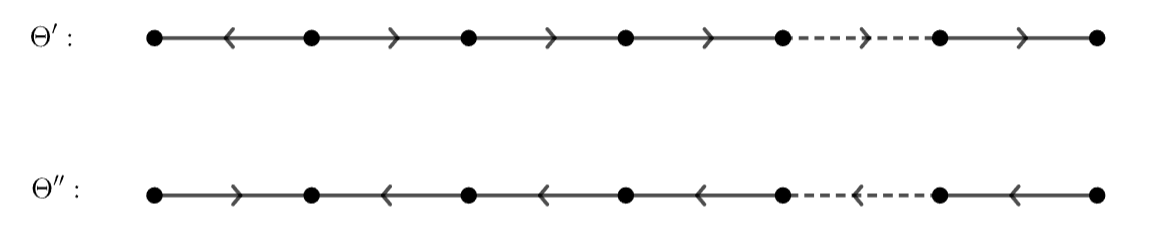}
    \caption{The $\Theta'$ and $\Theta''$ orientations on a path.}
    \label{fig:theta}
\end{figure}

These two orientations, together with the unidirectional orientation, are the only orientations for a path containing two vertices of distance $n-2$. 
\begin{lemma} \label{lem:n-2} Let $n\ge 3$. If $\overrightarrow{P_n}$ has two vertices of distance $n-2$, then the orientation of $\overrightarrow{P_n}$ is either unidirectional, $\Theta'$, or $\Theta''$. 
\end{lemma}

We shall use Lemma \ref{lem:n-2} to characterize $D$-antimagic paths when $D$ contains $0$ and $n-2$.
\begin{theorem} Let $n\ge 3$. If $\{0,n-2\}\subseteq D\subseteq \{0,1,2,\dots,n-2\}$, then $\overrightarrow{P_n}$ with $\Theta'$ orientation is $D$-antimagic.
\end{theorem}
\begin{proof}
Define a bijection $g_1:V(\overrightarrow{P_n})\rightarrow \{1,2,\dots,n\}$ such that $g(v_1)=1$ and $g(v_i)=n-i+2$, for $2\le i\le n$. We will show that $g$ is a $D$-antimagic labeling by considering two cases. 
\begin{description}
\item[Case 1 ($D=\{0,n-2\}$)] Based on $\Theta'$ orientation, the $D$-neighborhood and $D$-weight for each vertex is:
    \begin{description}
        \item[For $i=1$:] $N_D(v_1)=\{v_1\}$ and $\omega_D(v_1)=1$.
        \item[For $i=2$:] If $n=3$, $N_D(v_2)=\{v_1 ,v_2,v_3\}$ and $\omega_D(v_2)=6$. If $n\ge 4$, $N_D(v_2)=\{v_2,v_n\}$ and $\omega_D(v_2)=n+2$.
        \item[For $i=3,4,\dots,n$:] $N_D(v_i)=\{v_i\}$ and $\omega_D(v_i)=n-i+2$.    
    \end{description}
    Thus we have $\omega_D(v_1)<\omega_D(v_{i})<\omega_D(v_j)$, for $2\le i<j\le n$.
\item[Case 2 ($D\supset\{0,n-2\}$)] Here, $n\ge 4$.  Let $D=\{d_0,d_1,d_2,\dots,d_k\}$ be a distance set with $0=d_0<d_1<d_2<\dots<d_k=n-2$. 
\begin{description}
    \item[Case 2.1 ($d_1=1$).] We obtained the $D$-neighborhood and  $D$-weight for each vertex as follows:
    \begin{description}
        \item[For $i=1$:] $N_D(v_1)=\{v_1\}$ and $\omega_D(v_1)=1$.
        \item[For $i=2$:] $N_D(v_2)=\{v_1\}\cup \bigcup_{t=0}^{k}\{v_{d_t+2}\}$ and $\omega_D(v_2)=1+n(k+1)-\sum_{t=1}^{k}d_t$.
        \item[For $n-d_{t_0+1}+1\le i\le n-d_{t_0}\text{ and } 0\le t_0\le k-1$:] $N_D(v_i)=\bigcup_{t=0}^{t_0}\{v_{d_t+i}\}$ and 
        $\omega_D(v_i)=(n-i+2)(t_0+1)-\sum_{t=1}^{t_0}d_t$.
    \end{description} 
    Thus, $\omega_D(v_1)<\omega_D(v_n)<\omega_D(v_{n-1})<\dots<\omega_D(v_3)<\omega_D(v_2)$.

   \item[Case 2.2 ($d_1\ge 2$).]
   The $D$-neighborhood and $D$-weight for each vertex is: 
   \begin{description}
       \item[For $i=1$:] $N_D(v_1)=\{v_1\}$ and $\omega_D(v_1)=1$.
        \item[For $i=2$:] $N_D(v_2)=\bigcup_{t=0}^{k}\{v_{d_t+2}\}$ and $\omega_D(v_2)=n(k+1)-\sum_{t=1}^{k}d_t$.
        \item[For $n-d_{t_0+1}+1\le i\le n-d_{t_0},0\leq t_0 \leq k-1$:] $N_D(v_i)=\bigcup_{t=0}^{t_0}\{v_{d_t+i}\}$ and $\omega_D(v_i)=(n-i+2)(t_0+1)-\sum_{t=1}^{t_0}d_t$.
   \end{description} 
   And so $\omega_D(v_1)<\omega_D(v_n)<\omega_D(v_{n-1})<\dots<\omega_D(v_3)<\omega_D(v_2)$.
\end{description}\end{description}
In all the cases, the $D$-weight of each vertex is unique, and this concludes the proof. 
\end{proof}

  
\begin{theorem} Let $n\ge 3$. If $\{0,n-2\}\subseteq D\subseteq \{0,1,2,\dots,n-2\}$, then $\overrightarrow{P_n}$ with orientation $\Theta''$ is $D$-antimagic. 
\end{theorem}
\begin{proof}
Define a bijection $f:V(\overrightarrow{P_n})\rightarrow \{1,2,\dots,n\}$ by $f(v_i)=i$, for every $1\le i\le n$. We will show that $f$ is $D$-antimagic labeling on $\overrightarrow{P_n}$.
When $n=3$, $D={\{0,1\}}$, and $\omega_D(v_1)=3,\omega_D(v_2)=2,\omega_D(v_3)=5.$
For $n\ge4$, consider the following two cases: 
\begin{description}
    \item[Case 1 ($D=\{0,n-2\}$)] Based on $\Theta''$, the $D$-neighborhood and the $D$-weight for each vertex is:\\
    \textbf{For $1\le i\le n-1$:}
            $N_D(v_i)=\{v_i\}$ and $\omega_D(v_i)=i$.\\
    \textbf{For $i=n$:} $N_D(v_n)=\{v_2,v_n\}$ and $\omega_D(v_n)=n+2$.
   \item[Case 2 ($D\supset \{0,n-2\}$)] Let $D=\{d_0,d_1,d_2,\dots,d_k\}$, where $0=d_0<d_1<d_2<\dots<d_k=n-2$.
   \begin{description}
       \item[Case 2.1 ($d_1=1$)] Here is the $D$-neighborhood and distinct $D$-weight for each vertex in $\overrightarrow{P_n}$:
       \begin{description}
            \item[For $i=1$:] $N_D(v_1)=\{v_1,v_2\}$ and $ \omega_D(v_1)=3$.
            \item[For $i=2$:] $N_D(v_2)=\{v_2\}$ and $\omega_D(v_2)=2.$
            \item[For $3\le i\le d_2+1$:] $N_D(v_i)=\{v_i,v_{i-1}\}$ and $\omega_D(v_i)=2i-1 $.
            \item[For $d_{t_0-1}+2\le i\le d_{t_0}+1\text{ and }3\le t_0\le k$:] $N_D(v_i)=\bigcup_{t=0}^{t_0-1}\{v_{i-d_t}\}$ and $\omega_D(v_i)=it_0+\sum_{t=0}^{t_0-1}\left(d_t\right)$.
            \item[For $i=n$:] $N_D(v_n)=\bigcup_{t=0}^{k}\{v_{n-d_t}\}$ and $\omega_D(v_n)=\sum_{t=0}^{k}(n-d_t)$.
        \end{description}
      \item[Case 2.2 ($d_1\ge 2$).] Here is the $D$-neighborhood and the unique $D$-weight for each vertex.
      \begin{description}
            \item[For $1\le i\le d_1+1$:]$N_D(v_i)=\{v_i\}\text{ and }\omega_D(v_i)=i$.
            \item[For $d_{t_0-1}+2\le i\le d_{t_0}+1\text{ and }2\le t_0\le k$:] $N_D(v_i)=\bigcup_{t=0}^{t_0+1}\{v_{i-d_t}\}\text{ and }\omega_D(v_i)=it_0-\sum_{t=0}^{t_0-1}d_t$.
            \item[For $i=n$:]$N_D(v_n)=\bigcup_{t=0}^{k}\{v_{n-d_t}\}\text{ and }\omega_D(v_n)=\sum_{t=0}^{k}(n-d_t)=n(k+1)-\sum_{t=0}^{k}d_t$.
            \end{description}
  \end{description}
\end{description}
Since in all cases, the $D$-weight of each vertex is unique, then $f$ is $D$-antimagic. 
\end{proof}


The previous two theorems lead to the following characterization.
\begin{theorem}
Let $n\ge 3$ and $\{0,n-2\}\subseteq D\subseteq \{0,1,2,\dots,n-2\}$. Then, $\overrightarrow{P_n}$ is $D$-antimagic if and only if the orientation of $\overrightarrow{P_n}$ is either unidirectional, $\Theta'$, or $\Theta''$.
\end{theorem}

It is then natural to ask the characterization of $D$-antimagic paths, for other sets of $D$. 
\begin{problem}
    For any distance set $D$, find necessary and sufficient conditions for a path to admit a $D$-antimagic labeling.
\end{problem}

\section{Distance Antimagic Labelings on Oriented Trees} \label{sec:tree}


In this section, we characterize distance antimagic oriented trees. We start by observing the following properties of a tree.

\begin{observation}\label{ob:Tonesink} Let $T$ be a nontrivial oriented tree. Then,
    \begin{enumerate}
        \item $\overrightarrow{T}$ contains at least one sink.
        \item If $x$ is a vertex with the largest out-degree, then the number of sinks in $\overrightarrow{T}$ is at least $d^+(x)$.
    \end{enumerate}
\end{observation}


In the next lemma, we provide a necessary condition for $D$-antimagic oriented trees.
\begin{theorem}\label{lem:Tisunipath}
     If a tree $\overrightarrow{T}$ is distance antimagic then $\overrightarrow{T}$ is a unidirectional path.
\end{theorem}
\begin{proof}
Assume that $\overrightarrow{T}$ is not a unidirectional path. Based on the in-degree and out-degree of vertices in $T$, we investigate the following two cases.
\begin{description}
    \item[Case 1 ($\overrightarrow{T}$ has a vertex of in-degree two or more)] The in-neighbors have the same neighborhood.
    \item[Case 2 ($\overrightarrow{T}$ has a vertex of out-degree two or more)] Let such a vertex be $v$ and consider the paths starting from $v$. Each path will end at a sink. Thus, at least two sinks exist, both with zero $D$-weights.
\end{description}
Both cases lead to a contradiction, and so, all vertices of $\overrightarrow{T}$ have at most one in-degree or one out-degree. This implies that $\overrightarrow{T}$ is a unidirectional path.
\end{proof}

Finally, Theorem \ref{th:uni->minD<=1} and Theorem \ref{lem:Tisunipath} lead to the desired characterization.
\begin{theorem} \label{thm:DAtree}
An oriented tree is distance antimagic if and only if the tree is a unidirectional path.
\end{theorem}

However, the characterizations of the existence of $D$-antimagic trees for other distance sets are still largely unknown.
\begin{problem}
Let $D\neq \{1\}$ be a distance set. Find necessary and sufficient conditions for a tree to admit a $D$-antimagic labeling.
\end{problem}

\section{$D$-Antimagic Labeling on Oriented Linear Forests} \label{sec:forest}

Recall that a disjoint union of trees is a \textit{forest} and a disjoint union of paths is a \textit{linear forest}. We start our study of the existence of $D$-antimagic forests by observing the minimum value in the distance set $D$.

\begin{lemma} \label{lem:minforest} Let $F$ be an oriented forest.
If $F$ is $D$-antimagic then $\min(D)\le 1$. Furthermore, if $min(D)=1$ then $F$ consists of exactly one tree.
\end{lemma}
\begin{proof}
Assume that $min(D)\ge 2$. By Observation \ref{ob:Tonesink}, $F$ contains a sink. The the sink and its in-neighbors have zero $D$-weights. 

Lastly, let $min(D)=1$. Assume that $F$ consists of at least two trees. Again, by Observation \ref{ob:Tonesink}, $F$ consists of at least two sinks, all of which have zero $D$-weights.
\end{proof}

Now, for $D=\{1\}$, as a direct consequence of Lemma \ref{lem:minforest} and Theorem \ref{thm:DAtree}, we obtain the following characterization.
\begin{theorem}
An oriented forest is distance antimagic if and only if it consists of a single unidirectional path.
\end{theorem}
 

Next, we would like to characterize $D$-antimagic linear forests consisting of isomorphic paths. By Lemma \ref{lem:minD}, we know that for such linear forests to be $D$-antimagic, $min(D)=0$. For $m,n \geq 1$, let $\overrightarrow{mP_n}$ be the disjoint union of $m$ copies of $\overrightarrow{P_n}$, where $V(\overrightarrow{mP_n}):=\{v^j_i | 1 \leq i \leq n, 1 \leq j \leq m\}$. 

\begin{corollary} \label{cor:minD=0}
    If $\overrightarrow{mP_n}$ is $D$-antimagic, then $min(D)=0$.
\end{corollary}

Consider a unidirectional orientation $\Phi$ on $\overrightarrow{mP_n}$ as $\Phi :=\left\{\left(v_{i+1}^j, v_i^j\right)|1\le i\le n-1, 1\le j\le m\right\}$.
Under $\Phi$, the following $D$-neighborhoods, for $D=\{0\}$ and $D=\{k\}, 1\leq k\leq n-1$ are obtained.
\begin{description}
    \item[For $D=\{0\}$:] $N_D(v_i^j)=\{v_i^j\}$, for $1 \leq i\leq n$ and $1 \leq j\leq m$.
    \item[For $D=\{k\}$:] $N_D(v_i^j)= \{v_i^j\}$, for $1 \leq i \leq k$ and $1\leq j\leq m$,\\
    $N_D(v_i^j) = \{v_i^j, v_{i+1}^j\}$, for $k+1 \leq i \leq n$ and $1\leq j\leq m$.
\end{description}


Now we are ready to prove that if $|D|=2$, the linear forests consisting of isomorphic paths are $D$-antimagic.
\begin{theorem}\label{thm:mPn-0,k} Let $m,n \geq 1$ and $1\le k\le n-1$. A $\Phi$ oriented
$\overrightarrow{mP_n}$ is $\{0,k\}-$antimagic.
    \end{theorem}
    \begin{proof} Consider two cases based on the value of $k$.
    \begin{description}
        \item[Case 1 ($k=1$)] Define a bijection $h:V(\overrightarrow{mP_n})\rightarrow\{1, 2, \dots, mn\}$ where $h(v_i^j)=m(i-1)+j$ for $i=1, 2, \dots, n$ and $j=1, 2, \dots, m$.  We will prove that $h$ is $D$-antimagic.
        Under $\Phi$, the $D$-neighborhood and $D$-weight for each vertex are as follows.
        \begin{description}
            \item[For $i=1$:]  $N_D(v_1^j)=\left\lbrace v_1^j\right\rbrace$ and  $\omega_D(v_1^j)=j.$ 
            \item[For $2\le i\le n$:] $N_D(v_i^j)=\left\lbrace v_i^j, v_{i-1}^j\right\rbrace$ and $\omega_D(v_{i}^{j}) = 2j+m(2i-3)$.
        \end{description}
        It can be shown that the $D$-weights form a strictly increased sequence where 
        $\omega_D(v_{i}^{j})< \omega_D(v_{i}^{j+1})$, for $1\le j\le m-1$, and $\omega_D(v_{i}^{m}) < \omega_D(v_{i+1}^{1})$, for $1\le i\le n-1$. 
\item[Case 2 ($k=2,3,\dots,n-1$)] For $j=1, 2, \dots, m$, the $D$-neighborhood and $D$-weight of each all vertices are as follows:
\begin{description}
    \item[For $1 \leq i \leq k$:] $N_D(v_i^{j})=\{v_i^{j}, v_{i-k}^{j}\}$ and $\omega_{D}(v_i^{j})=m(i-1)+j$.
    \item[For $k+1 \leq i \leq n$:] $N_D(v_i^{j})=\{v_i^J\}$ and $\omega_{D}(v_i^{j})= 2j+(2i-k-2)m.$
\end{description}

    It can be shown that the $D$-weights form a strictly increased sequence where         $\omega_D(v_{i}^{j})< \omega_D(v_{i}^{j+1})$, for $1\le j\le m-1$, and $\omega_D(v_{i}^{m}) < \omega_D(v_{i+1}^{1})$, for $i = 1, 2, \dots, n-1.$ 
    \end{description}
    Therefore, each vertex has a distinct $D$-weight and $h$ is a $D$-antimagic labeling. 
    \end{proof}

Figure \ref{fig:mpn01} provides an example of the labeling in Theorem \ref{thm:mPn-0,k}.

\begin{figure}[h!]
    \centering    \includegraphics[width=0.6\linewidth]{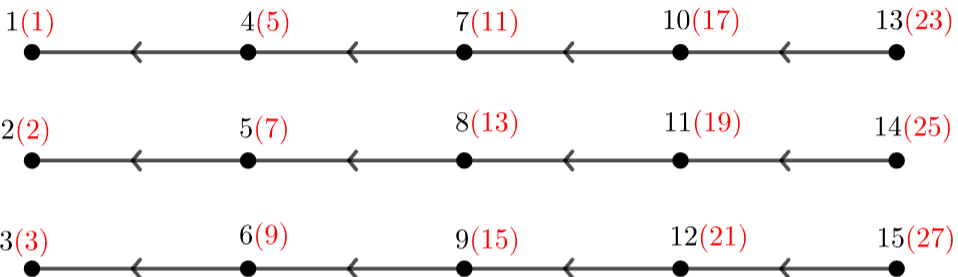}
    \caption{A \{0,1\}-antimagic labeling on $\Phi$ oriented $\overrightarrow{3P5}$.}
    \label{fig:mpn01}
\end{figure} 

Theorem \ref{thm:mPn-0,k} and the fact that the only orientation admits distance $n-1$ is unidirectional lead us to characterize the $\{0,n-1\}$-antimagic $\overrightarrow{mP_n}$.
\begin{corollary}
Let $m,n \geq 1$. $\overrightarrow{mP_n}$ is $\{0,n-1\}$-antimagic if and only if $\overrightarrow{mP_n}$ is unidirectional.
\end{corollary}

In the next theorem, we characterize $D$-antimagic unidirectional $\overrightarrow{mP_n}$, by showing that, for this particular case, the sufficient condition in Corollary \ref{cor:minD=0} is also necessary.

\begin{theorem} 
Let $m,n\ge 2$. There exists an orientation such that $\overrightarrow{mP_n}$ is $D$-antimagic if and only if $\min{(D)}=0$.
\end{theorem}
\begin{proof} 
By  Corollary \ref{cor:minD=0}, we only need to prove that if $\min(D)=0$, then there exists a $D$-antimagic labeling on unidirectional $\overrightarrow{mP_n}$. 

If $D=\{0\}$, see Observation \ref{obs:0-antimagic}, and if $|D|=2$, see Theorem \ref{thm:mPn-0,k}. For $|D|\ge 3$, let $D=\{d_0, d_1, d_2,\dots, d_k\}$ where $0=d_0<d_1<d_2<\dots<d_k\le n-1$. Choose the unidirectional orientation $\Phi$ and define a bijection $g: V(\overrightarrow{mP_n})\rightarrow\{1, 2, \dots, mn\}$ by $g(v_i^j)=j+m(i-1)$ for $i=1, 2, \dots, n$ and $j=1,2,\dots,m.$

For $1\le j\le m$, under $g$, the $D$-neighborhood and $D$-weight of all vertices are:
\begin{description}
    \item[For $1\le i\le d_1$:] $N_D(v_i^j)=\{v_i^j\}$ and $\omega_D(v_i^j)=j+m(i-1)$.
    \item[For $d_{t_0}+1\le i\le d_{t_0+1},1\le t_0\le k-1$:] $N_D(v_i^j)=\{v_{i-d_t}^j|0 \le t \le t_0\}$ and \\$\omega_D(v_i^j)=(t_0+1)(n-i+1)-\sum_{t=0}^{t_0} d_t$.
    \item[For $d_k+1\le i\le n$:] $N_D(v_i^j)= \{v_{i-d_t}^j|0 \le t \le k\}$ and $\omega_D(v_i^j)=(k+1)(n-i+1)-\sum_{t=0}^{k} d_t$.
\end{description}
For $j'>j$ and arbitrary $i$ and $i'$, $\omega_D(v_{i'}^{j'})>\omega_D(v_i^j)$. For $i'>i$ and $j=j'$, $\omega_D(v_{i'}^{j'})>\omega_D(v_i^j)$. This completes the proof.
\end{proof}


Our last result is a $D$-antimagic labeling for linear forests that consist of arbitrary paths. Here we provide the notation for such linear forests. For $j=1,\dots,t$, where $n_j<n_{j'}$, if $j<j'$, let $\mathcal{F}=\bigcup_{j=1}^t m_j\overrightarrow{P_{n_j}}$ be an oriented linear forest with vertex set $$V(\mathcal{F})=\left\{v_i^{j,s}|1\le j\le t,1\le s\le m_j, 1 \le i \le n_j\right\}.$$ 
Set a unidirectional orientation on $\mathcal{F}$ by $A\left(\mathcal{F}\right)=\left\{\left(v_i^{j,s}, v_{i-1}^{j,s}\right)|1\le j\le t, 1\le s\le m_j, \ 2\le i\le n_j\right\}.$
Let $1 \le k \le n_1-1$. For all $j \in [1,t]$ and $s \in [1,m_j]$, the $\left\{k\right\}$-neighborhood of each vertex in $\mathcal{F}$ is $N_{\left\{k\right\}}\left(v_i^{j,s}\right)=\emptyset$, for $1 \le i \le k$, and $N_{\left\{k\right\}}\left(v_i^{j,s}\right)=\left\{v_{i-k}^{j,s}\right\}$, for $k+1\le i\le n_j$. 


\begin{theorem} \label{thm:F}
 A unidirectional $\mathcal{F}$ is $\left\{0,1\right\}$-antimagic. 
\end{theorem}
\begin{proof} The $D$-neighborhoods of each vertex of $\mathcal{F}$ are 
$N_D\left(v_1^{j,s}\right)=\left\{v_1^{j,s}\right\}$, for $1 \le j \le t, 1 \le s \le m_j$, and
$N_D\left(v_i^{j,s}\right)=\left\{v_i^{j,s}, v_{i-1}^{j,s}\right\}$, for $1\le j\le t,1\le s\le m_j,2\le i\le n_j.$

For all $1\le j\le t, 1\le s\le m_j,$ and $1\le i\le n_j$, define a mapping $f_*:V\left(\mathcal{F}\right) \rightarrow \{1, 2, \dots,\sum_{p=1}^t m_pn_p\}$ by 
$$f_*\left(v_i^{j,s}\right)=\sum_{q=1}^{j_0-1}\sum_{p=q}^t m_p\left(n_q-n_{q-1}\right)+\left(i-n_{j_0-1}-1\right)\sum_{q=j_0}^t m_q+\sum_{q=j_0}^{j-1} m_q +s$$
where $n_0=0$ and $j_0$ is a natural number such that $n_{j_0-1}+1\le i\le n_{j_0}$. 

First, we will show that $f_*$ is a bijection. 
Let $v_i^{j,s},v_{i'}^{j',s'}$ be two distinct vertices in $\mathcal{F}$. This means either $j \ne j'$ or $i \ne i'$ or $s \ne s'$. We will investigate each of the three possible inequalities.
\begin{description}
\item[Case 1 ($i \ne i'$):] Without loss of generality, suppose $i'>i$. If $n_{j_0'-1}+1\le i'\le n_{j_0'}$, then 
\begin{align*}
 f_*\left(v_i^{j,s}\right)=&\sum_{q=1}^{j_0-1} \left(\sum_{p=q}^t m_p\right) \left(n_q - n_{q-1}\right)+\left(i-n_{j_0-1}-1\right) \sum_{q=j_0}^t m_q + \sum_{q=j_0}^{j-1} m_q +s\\
&\le \sum_{q=1}^{j_0-1} \left(\sum_{p=q}^t m_p\right) \left(n_q - n_{q-1}\right)+\left(i-n_{j_0-1}\right) \sum_{q=j_0}^t m_q\\
 &<\sum_{q=1}^{j_0'-1} \left(\sum_{p=q}^t m_p\right) \left(n_q - n_{q-1}\right)+\left(i'-n_{j_0'-1}-1\right) \sum_{q=j_0'}^t m_q + \sum_{q=j_0'}^{j'-1} + s'=f_*\left(v_{i'}^{j',s'}\right).
\end{align*}
\item[Case 2 ($j\ne j'$ and $i=i'$):] Without lost of generality, let $j'>j$. 
Then,
\begin{align*}
    f_*&\left(v_i^{j,s}\right)=\sum_{q=1}^{j_0-1} \left(\sum_{p=q}^t m_p\right) \left(n_q - n_{q-1}\right)+\left(i-n_{j_0-1}-1\right) \sum_{q=j_0}^t m_q + \sum_{q=j_0}^{j-1} m_q +s\\
    &\le \sum_{q=1}^{j_0-1} \left(\sum_{p=q}^t m_p\right) \left(n_q - n_{q-1}\right)+\left(i'-n_{j_0-1}-1\right) \sum_{q=j_0}^t m_q + \sum_{q=j_0}^{j} m_q\\
    &<\sum_{q=1}^{j_0-1} \left(\sum_{p=q}^t m_p\right) \left(n_q - n_{q-1}\right)+\left(i'-n_{j_0-1}-1\right) \sum_{q=j_0}^t m_q + \sum_{q=j_0}^{j'-1} m_q + s'=f_*\left(v_{i'}^{j',s'}\right).
\end{align*} 
\item[Case 3 ($s \ne s'$, $j=j'$, and $i=i'$):] 
It is obvious that $f_*\left(v_i^{j,s}\right)\ne f_*\left(v_{i'}^{j',s'}\right)$.
\end{description}
In all cases we obtain $f_*\left(v_i^{j,s}\right)\ne f_*\left(v_{i'}^{j',s'}\right)$, which means 
$f_*$ is an injection, and so $f_*$ is a bijection.

As a last step, we will show that each $D$-weight is distinct. Under $f_*$, the $D$-weight for each vertex is: 
\begin{align*}
\omega_D\left(v_1^{j,s}\right)&=s+\sum_{q=1}^{j-1}m_q \text{ for } 1\le j\le t,1\le s\le m_j\text{ and }\\
\omega_D\left(v_i^{j,s}\right)&=2\sum_{q=1}^{j_0-1}\left(\sum_{p=q}^{t}m_p\right)\left(n_q-n_{q-1}\right)+\left(2i-2n_{j_0-1}-3\right)\sum_{q=j_0}^{t}m_q+2\sum_{q=j_0}^{j-1}m_q+2s \\ &\text{ for } n_{j_0-1}+1\le i\le n_{j_0},i\ne 1,n_0=0,1\le j_0\le t.
\end{align*}
        
Let $v_{i,s}^j$ and $v_{i',s'}^{j'}$ be any vertices in $\mathcal{F}$. If not specified, all these indices (i.e. $i,s,j, i',s',j'$) are taken arbitrarily, where $1\le i,i'\le n_j;1\le s,s'\le m_j;$ and $1\le j,j'\le t$. Consider the following cases.
\begin{description}
    \item[Case 1 ($i\ne i'$):] If $i>i'$ then $\omega_D(v_{i,s}^{j})<\omega_D(v_{i',s'}^{j'})$.
    \item[Case 2 ($i=i'$):] Consider two subcases based on $j$.
    \begin{description}
        \item[Subcase 2.1 ($j\ne j'$):] If $j<j'$ then $\omega_D(v_{i,s}^{j})<\omega_D(v_{i,s'}^{j'})$.
        \item[Subcase 2.2 ($j= j'$):] If $s<s'$ then $\omega_D(v_{i,s}^{j})<\omega_D(v_{i,s'}^{j})$.
    \end{description}
\end{description}
Thus, $f_*$ is a $\{0,1\}$-antimagic labeling of $\mathcal{F}$. 
\end{proof}

Figure \ref{fig:F0,1} provides an example of a labeling in Theorem \ref{thm:F}.
\begin{figure}[h!]
    \centering    \includegraphics[width=0.7\linewidth]{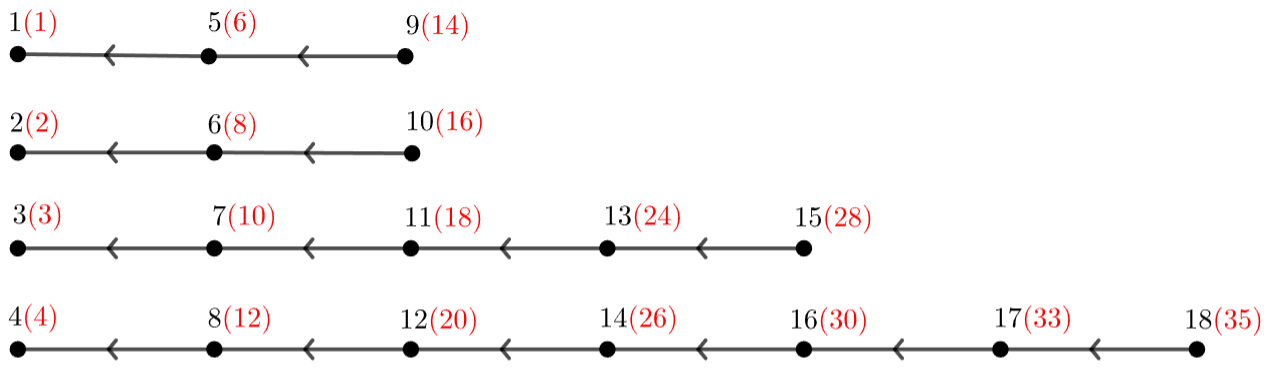}
    \caption{A $\{0,1\}$-antimagic labeling on $\mathcal{F}=2\overrightarrow{P_3}\cup \overrightarrow{P_5}\cup \overrightarrow{P_7}$. The black numbers are the labels and the red numbers inside the brackets are the $\{0,1\}$-weights.}
    \label{fig:F0,1}
\end{figure}

Theorem \ref{thm:F} provides a $D$-antimagic labeling on unidirectional $\mathcal{F}$ where $D=\{0,1\}$. Therefore, it is natural to propose the following open problem.
\begin{problem}
Find all sets $D$ and orientation $\Omega$ such that $\mathcal{F}$ with orientation $\Omega$ is $D$-antimagic.
\end{problem}

\section*{Acknowledgement} 
This research has been partially supported by Riset Unggulan ITB 2023.


\end{document}